\documentclass[11pt]{article}
\usepackage{amsthm, amsmath, amssymb, amsfonts, url, booktabs, tikz, setspace, fancyhdr, bm}
\usepackage[hidelinks]{hyperref}
\usepackage[margin = 1in]{geometry}
\usepackage{hyperref, enumerate}
\usepackage[shortlabels]{enumitem}
\usepackage[babel]{microtype}
\usepackage[english]{babel}
\usepackage[capitalise]{cleveref}
\usepackage{comment}
\usepackage{scalerel}
\usetikzlibrary{arrows.meta}

\counterwithin{figure}{section}

\newtheorem{theorem}{Theorem}[section]

\newtheorem{lemma}[theorem]{Lemma}
\newtheorem{corollary}[theorem]{Corollary}
\newtheorem{conjecture}[theorem]{Conjecture}
\newtheorem{claim}[theorem]{Claim}

\theoremstyle{definition}

\theoremstyle{remark}
\newtheorem*{remark}{Remark}

\DeclareMathOperator{\Bin}{Bin}

\newcommand{\EE}{\mathbb{E}}
\newcommand{\PP}{\mathbb{P}}

\newcommand{\sgn}{\mathrm{sgn}}
\newcommand{\Bi}{\mathrm{Bin}}
\newcommand{\pp}{\scaleto{++}{4pt}}
\newcommand{\mm}{\scaleto{--}{4pt}}
\newcommand{\mpl}{\scaleto{-+}{4pt}}
\newcommand{\pmi}{\scaleto{+-}{4pt}}

\newcommand{\eps}{\varepsilon}

\tikzstyle{p}+=[fill=black, circle, minimum width = 1pt, inner sep =
1pt]
\tikzstyle{w}+=[fill=white, draw, circle, minimum width = 1pt, inner sep =
1.5pt]

\begin{document}

\title{Majority dynamics on sparse random graphs}

\author{
Debsoumya Chakraborti\thanks{
Discrete Mathematics Group, Institute for Basic Science (IBS), Daejeon, Republic of Korea. Supported by the Institute for Basic Science (IBS-R029-C1).
E-mail: {\tt
debsoumya@ibs.re.kr}.}\and 
Jeong Han Kim\thanks{School of Computational Sciences, Korea Institute for Advanced Study (KIAS), Seoul, Republic of Korea. Supported in part by National Research Foundation of Korea (NRF) Grants funded by the Korean Government (MSIP) (NRF-2012R1A2A2A01018585 \& 2017R1E1A1A03070701) and by KIAS Individual Grant (CG046001).
%at Korea Institute of Advanced Study. %This work was partially carried out while the author were visiting the Discrete Mathematics Group of Institute for Basic Science (IBS), Daejeon,  Republic of Korea.
}\and
Joonkyung Lee\thanks{
Department of Mathematics, University College London, Gower Street, London WC1E 6BT, UK.
E-mail: {\tt
joonkyung.lee@ucl.ac.uk}. Supported by IMSS Research Fellowship.}\and
Tuan Tran\thanks{
Discrete Mathematics Group, Institute for Basic Science (IBS), Daejeon, Republic of Korea. Supported by the Institute for Basic Science (IBS-R029-Y1).
E-mail: {\tt
tuantran@ibs.re.kr}.}
}

\date{}

\maketitle

\begin{abstract}
\emph{Majority dynamics} on a graph $G$ is a deterministic process such that every vertex updates its $\pm 1$-assignment according to the majority assignment on its neighbor simultaneously at each step.
Benjamini, Chan, O'Donnell, Tamuz and Tan conjectured that, in the Erd\H{o}s--R\'enyi random graph $G(n,p)$, the random initial $\pm 1$-assignment converges to a $99\%$-agreement with high probability whenever $p=\omega(1/n)$.

This conjecture was first confirmed for $p\geq\lambda n^{-1/2}$ for a large constant $\lambda$ by Fountoulakis, Kang and Makai. Although this result has been reproved recently by Tran and Vu and by Berkowitz and Devlin, it was unknown whether the conjecture holds for $p< \lambda n^{-1/2}$.
We break this $\Omega(n^{-1/2})$-barrier by proving the conjecture for sparser random graphs $G(n,p)$, where $\lambda' n^{-3/5}\log n \leq p \leq \lambda n^{-1/2}$ with a large constant $\lambda'>0$.
\end{abstract}

\section{Introduction}\label{sec:intro}
\emph{Majority dynamics} on a graph $G$ is a fundamental example of opinion exchange dynamics that models human interactions in a society. 
Formally, every vertex $v\in V(G)$ has its \emph{opinion} $s_t(v)$ on Day $t$, where each $s_t(v)$ updates simultaneously by the majority opinion on the neighbors at each day.
That is, for $t\geq 1$,
\begin{align*}
    s_t(v)=\begin{cases}
    \sgn \sum_{u\sim v}s_{t-1}(u) ~~\text{ if }\sum_{u\sim v}s_{t-1}(u)\neq 0,\\
    s_{t-1}(v) ~~\text{ otherwise }
    \end{cases}
\end{align*}
and the \emph{initial opinions} $s_0(v)$ are given.
This model has been studied in various areas, including combinatorics~\cite{BCOTT16,BD20,FKM20,GO80,TV19}, psychology~\cite{CH56psy} and biophysics~\cite{McP90}, since 1940s.
For more discussions on relevant models, we refer the reader to the survey~\cite{MT17}.

In the study of majority dynamics, perhaps one of the most natural questions is what happens after sufficiently many days. 
For every finite graphs $G$, Goles and Olivos~\cite{GO80} showed that each $s_t(v)$ always converges to a periodic behavior of length at most two, no matter what the initial opinion~$s_0$ is.
In other words, the dynamics eventually either alternates between two distinct states or converges to a single state.
Particularly interesting examples of the single state may be an \emph{$(1-\varepsilon)$-proportion agreement} or \emph{unanimity},
i.e., 
$|\sum s_t(v)|\geq (1-2\varepsilon)n$ or $|\sum s_t(v)|=n$, respectively.

The next natural question is then under what circumstances majority dynamics converges to a single state. In particular, when does unanimity (or 99\% agreement) appear?
In the view of probabilistic combinatorics, the most popular host graph $G$ may be the Erd\H{o}s--R\'enyi random graph $G(n,p)$, where each edge on the vertex set $[n]:=\{1,2,\cdots,n\}$ exists with probability~$p$ independently at random.
The first in-depth study in this direction was done by Benjamini, Chan, O'Donnell, Tamuz and Tan~\cite{BCOTT16}, where they proposed the following intriguing conjecture.

\begin{conjecture}[{\cite[Conjecture 1.5]{BCOTT16}}]\label{conj:main}
Let $s_0(v)$ be sampled uniformly at random for each $v\in [n]$ and let $\varepsilon\in(0,1]$ be given.
Then with probability $1-\varepsilon$, the vertices in $G(n,p)$ have an $(1-\varepsilon)$-proportion agreement $|\sum_v s_0(v)|\geq (1-2\varepsilon)n$ after sufficiently many days whenever $p=\omega(1/n)$.
\end{conjecture}

In~\cite{BCOTT16}, the authors also conjectured that the converse of the statement above is true. That is, given any fixed $C>0$, $G(n,C/n)$ eventually oscillates between two states with probability $1-o(1)$. 
They actually gave positive evidences for both conjectures. First, it is proved in~\cite[Theorem~3]{BCOTT16} that a random 4-regular graph never converges to the single state with probability $1-o(1)$.
Second, as a partial progress towards~\cref{conj:main}, \cite[Theorem~2]{BCOTT16} shows that, under the stronger assumption $p\geq \lambda n^{-1/2}$ for a large constant $\lambda$, the probability that unanimity appears is at least $0.4$. Indeed, the probability bound here is weaker than the conjectured value $1-\varepsilon$, whereas unanimity is a slightly stronger condition than the $(1-\varepsilon)$-proportion agreement.
This was strengthened by Fountoulakis, Kang and Makai~\cite[Theorem~1.1]{FKM20}, who pushed the probability bound $0.4$ in \cite[Theorem~2]{BCOTT16} to $1-\varepsilon$ and confirmed \cref{conj:main} under the condition $p\geq\lambda n^{-1/2}$ instead of $p=\omega(1/n)$, where $\lambda$ depends on $\varepsilon$.

\medskip

There are other models with various alternative settings for the initial opinion $s_0$ or the host graphs~$G$, e.g., on pseudorandom graphs~\cite{Z18}, with linear bias on $s_0$~\cite{GZ18}, or on grids (or tori) $G$~\cite{GZ17}.
One of the most notable variants may be the one suggested by Tran and Vu~\cite{TV19}, where $s_0$ is randomly chosen while the discrepancy between the number of vertices with distinct $s_0$-values is fixed, i.e., there are $\lceil n/2\rceil +C$ vertices $v$ with $s_0(v)=+1$ for a fixed number $C\geq 0$. Note that, in contrast, the independent random initial assignment gives $\Omega(\sqrt{n})$ bias in either direction with probability $1-\varepsilon$, as will be proved in~\cref{lem:deviation}.
Tran and Vu proved that $C=6$ is enough to force unanimity on the random graph $G(n,1/2)$ with probability strictly larger than $.51$ and, in the same paper, reproved the Fountoulakis--Kang--Makai theorem.

Very recently, Berkowitz and Devlin~\cite{BD20} studied the Tran--Vu model further. They again reproved the Fountoulakis--Kang--Makai theorem by using their ``Central Limit Theorem" and also lowered the constant discrepancy bound $C=6$ by Tran and Vu to $C=2$.

\medskip

Despite these two alternative proofs of the Fountoulakis--Kang--Makai theorem and deeper studies on somewhat ``sharper" models, nobody ever managed to settle~\cref{conj:main} beyond the barrier $p\geq \lambda n^{-1/2}$. To quote very recent work~\cite{CDF20} in the area, ``the study of
majority dynamics for $p=o(n^{-1/2})$ imposes immense complications."
Our main result is to confirm~\cref{conj:main} for sparser random graphs $G(n,p)$ with $\lambda' n^{-3/5}\log n\leq p\leq \lambda n^{-1/2}$, thereby breaking the barrier for the first time.
\begin{theorem}\label{thm:main}
Let $s_0(v)$ be sampled uniformly at random for each $v\in [n]$ and let $\varepsilon\in(0,1]$ and $\lambda>0$ be given.
Then there exist $n_0$ and $\lambda'$ such that, with probability at least $1-\varepsilon$, the vertices in $G(n,p)$ reach the unanimous state $\sgn \sum_v s_0(v)$  after %sufficiently many days 
six days whenever $\lambda' n^{-3/5}\log n\leq p\leq \lambda n^{-1/2}$ and $n\geq n_0$.
\end{theorem}

Together with the Fountoulakis--Kang--Makai theorem, this extends the range where~\cref{conj:main} is settled further to $p\geq \lambda' n^{-3/5}\log n$.
%Although this is already interesting in its own right, 
However, the full~\cref{conj:main} still remains open and there seem to be a substantial amount of technical obstacles to overcome, which will be discussed in due course.

\medskip

Somewhat analogously to~\cite{BD20,TV19}, our strategy is to analyze the dynamics under the ``sharpest" initial setting carefully. Let $r_0$ be a $\pm 1$-assignment on $[n]$ obtained by choosing $\lceil\frac{n}{2}\rceil$ vertices $v$ uniformly at random to assign $+1$ while giving $-1$ to the remaining $\lfloor\frac{n}{2}\rfloor$ vertices and let $r_t$, $t>0$, be the Day $t$ opinion resulting from majority dynamics with the initial opinion $r_0$.
Then the uniform random choice of~$s_0$ resembles an alteration of $r_0$ obtained by turning $\Omega(\sqrt{n})$ $-1$'s to $+1$'s, which produces more $+1$'s on Day~1 and hence proves a nontrivial shift on $\sum_v s_2(v)$.
To chase this effect of the alteration, we call a vertex $v$ \emph{$\gamma$-almost-positive} if $\sum_{w \in N(v)} r_1(w) > - \gamma p^{3/2}n$, which are ``potentially positive" vertices in a rough sense.
Arguably the following is our key lemma, which shows that slightly more than a half of the vertices are $\gamma$-almost-positive with probability $1-o(1)$.
\begin{lemma}\label{lem:day0_main}
For $\lambda>0$, there exists $\lambda'>0$ such that the following holds: 
For every $\gamma>0$, there is $\alpha>0$ such that
 the number of $\gamma$-almost-positive vertices in $G=G(n,p)$ with $\lambda'n^{-3/5}\log n \leq p\leq \lambda n^{-1/2}$ is at least $\frac{n}{2} + \alpha pn^{3/2}$ with probability $1-o(1)$.
\end{lemma}

%\medskip

This paper is organized as follows. In \cref{sec:prelim}, we give some basic definitions and tools, which may be skipped by experienced readers. The proof of \cref{thm:main} will be given throughout \cref{sec:day0,sec:day1+}. In particular, \cref{sec:day0} contains the main new ideas to analyze the first two days, including the proof of~\cref{lem:day0_main} at the end of the section. 
In \cref{sec:day1+}, the shift of~$s_2$ obtained by using~\cref{lem:day0_main} will show that unanimity must appear by Day 6.

\section{Preliminaries}\label{sec:prelim}

An event $A_n$ that depends on the parameter $n$ occurs \emph{with high probability} (or briefly, \emph{w.h.p.}) if $\PP[A_n]$ tends to $1$ as $n$ tends to infinity. The notation $x=a \pm b$ means the inequality $a-b\leq x\leq a+b$.
We use the standard asymptotic notation such as $O,o,\Omega,\omega$ and $\Theta$ to avoid carrying numerous constants when estimating nonnegative functions.
For instance, $y_n=x_n\pm O(n)$ means that $|y_n-x_n|=O(n)$.
In addition, $f(n)\gg g(n)$ for nonnegative $f$ and $g$ means $\lim_{n\rightarrow\infty} g(n)/f(n)=0$.
The parameter~$n$ that represents the number of vertices in $G(n,p)$ will be assumed to be large enough whenever necessary. Logarithms will always be understood to be base $e$. We denote by $\Bi(n,p)$ the binomial distribution with $n$ independent trials of one-probability $p$. 

\medskip

One of the most frequently used probabilistic tools in what follows is the Chernoff bound, due to Chernoff~\cite{Ch52} and to Okamoto~\cite{Ok59}.
We use the version stated by Janson~\cite[Theorem~1]{J02}.
\begin{lemma}[The Chernoff bound]\label{lem:chernoff}
Let $X =\sum_{i=1}^n X_i$, where $X_i$ are independent Bernoulli variable with $\PP[X_i=1]=p_i$. 
Let $\mu=\EE[X]=\sum_{i=1}^n p_i$. Then for $t\geq 0$,
\begin{enumerate}[(i)]
    \item $\PP[X\geq \mu+t]\leq e^{-\frac{t^2}{2\mu+2t/3}}$ and
    \item $\PP[X\leq \mu-t]\leq e^{-\frac{t^2}{2\mu}}$.
\end{enumerate}
\end{lemma}

An easy consequence of the Chernoff bound is the following concentration result.

\begin{lemma}\label{lem:whp}
Let $X\sim \Bi(n,p)$ and let $\lambda>0$. %with $p\geq 1/n$.
Then there exists a constant $C>0$ such that, for $n$ large enough,
\begin{enumerate}[(i)]
    \item  If $p\geq \lambda/n$, then $\PP\big[ X= np\pm C\sqrt{np}\log n\big] \geq 1-1/n^6$ and
    \item If $p\leq \lambda /n$, then $\PP[X\leq C\log n]\geq 1-1/n^3$.
\end{enumerate}
\end{lemma}
\begin{proof}
(i) Let $t=C\sqrt{np}\log n$ in~\cref{lem:chernoff}. Then for large enough constant $C$,
\begin{align*}
    \PP[|X-np|\geq t]&=\exp\left(-\frac{C^2np\log^2 n}{2np+\frac{2}{3}C\sqrt{np}\log n}\right)\leq \exp\left(-\frac{C^2\sqrt{\lambda}\log^2 n}{2\sqrt{\lambda}+\frac{2}{3}C\log n}\right)\leq 1/n^6.
\end{align*}
(ii) Again, for large $C$,
\begin{align*}
    \PP[X\geq C\log n]\leq \PP[X\geq np+(C-\lambda)\log n]\leq \exp\left(-\frac{(C-\lambda)^2\log^2 n}{2\lambda+\frac{2}{3}(C-\lambda)\log n}\right)\leq 1/n^3.\tag*{\qedhere}
\end{align*}
\end{proof}

\medskip

We use the following version of the Berry--Esseen inequality to estimate possibly non-binomial distributions.
\begin{theorem}[see, e.g., \cite{BE56}] \label{berry}
There is a universal constant $C_0$ such that the following holds: 
Let $X_1, X_2, \ldots, X_n$ be independent random variables with zero mean, variances $\EE(X_i^2)=\sigma_i^2 > 0$, and the absolute third moments $\mathbb{E}(|X_i|^3) = \rho_i < \infty$. 
Then 
$$\sup_{x \in \mathbb{R}} \left|\PP\left[\frac{\sum_{i=1}^n X_i}{\sigma_X} \le x\right] - \Phi(x)\right| \le C_0 \sigma_X^{-3} \sum_{i=1}^n \rho_i,$$
where $\sigma_X^2$ is the variance of $X = \sum_{i=1}^n X_i$, i.e., $\sigma_X^2 = \sum_{i=1}^n \sigma_i^2$, and $\Phi(x)$ is the cumulative distribution function of the standard normal variable.
In particular, if $|X_i|\leq M$ for an absolute constant $M>0$ almost surely, then the RHS is $O(\sigma_X^{-1})$.
\end{theorem}

When using the Berry--Esseen bound, we need some simple facts about the function $\Phi$.

\begin{lemma}\label{lem:psi}
Let $\Psi(x):=1-\Phi(x)$ be the probability that a standard normal variable takes value higher than $x$,
i.e., $\Psi(x)=\frac{1}{\sqrt{2\pi}}\int_{x}^{\infty} e^{-t^2/2}dt$. Then
\begin{enumerate}[(i)]
    \item $\Psi$ is a contraction, i.e., $|\Psi(x)-\Psi(y)|\le |x-y|$;
    \item for $x,y$ with $x+y<0$ and $|x|+|y|\leq c$ for a constant~$c$, there exists $C>0$ only depending on $c$ such that $\Psi(x)+\Psi(y) \geq 1-C(x+y)$.
\end{enumerate}
\end{lemma}
\begin{proof}
(i) %Immediately follows from Mean Value Theorem.
 $|\Psi(x)-\Psi(y)|=\left|\frac{1}{\sqrt{2\pi}}\int_{x}^{y}e^{-t^2/2}\,dt\right|\le |x-y|$.

\noindent (ii) 
Let $C=(2\pi)^{-1/2}e^{-c^2/2}$. Then 
\begin{align*}
    \Psi(x)+\Psi(y)=1+\frac{1}{\sqrt{2\pi}}\int_{x}^{-y}e^{-t^2/2}\,dt \geq 1+\int_{x}^{-y}C\,dt=1-C(x+y).\tag*{\qedhere}
\end{align*}
\end{proof}
The rest of this section includes various lemmas that approximate binomial variables to one another.
We omit the proof of the first lemma below, which can be found, e.g., in~\cite[Lemma 9]{BD20}.

\begin{lemma}\label{lem:binom_shift}
Let $X\sim \Bin(n,p)$ and $Y\sim \Bin(m,p)$ be independent. There is a universal constant $C$
(independent of $m, n$, and $p$) such that for all $t$ we have
\begin{align*}
\big|\PP[X-Y=t+1] - \PP[X-Y=t]\big| \le \frac{C}{(m+n)p(1-p)}.
\end{align*}
\end{lemma}
%\noindent For a proof, see, e.g.,
\begin{lemma}\label{lem:Binomial=0}
Let $X\sim \Bi(n,p)$ and $Y\sim \Bi(m,p)$ be independent. 
Suppose that $\frac{1}{n}\ll p\ll \frac{1}{\log n}$ and that $|n-m|\leq \sqrt{n\log n}$. 
%Let $X:=\sum_{i=1}^n X_i$ and $Y:=\sum_{j=1}^m Y_j$. 
Then $\PP[X=Y]=\Theta\big(\frac{1}{\sqrt{np}}\big)$ and $\PP[X\geq Y]=\frac{1}{2}\pm O\big(\frac{1+|n-m|p}{\sqrt{np}}\big)$.
\end{lemma}

\begin{proof}
Since $X-Y$ has mean $(n-m)p$ and standard deviation $\sigma=\sqrt{(n+m)p(1-p)}$, it follows from Chebyshev's inequality that
\begin{align}\label{eq:P>1/4}
    \PP\left[|X-Y -(n-m)p|\leq 2\sigma\right]\geq 3/4.
\end{align}
Now let $r:=\max_j\PP[X-Y=j]$. Then \eqref{eq:P>1/4} implies $r=\Omega(\frac{1}{\sigma})=\Omega(\frac{1}{\sqrt{np}})$.
The probability mass function of $X-Y$ is unimodal and attains the maximum at one of the two closest integers to the mean $(n-m)p$.
Thus, by \cref{lem:binom_shift}, $\PP[X=Y]\geq r-C\cdot \frac{|n-m|p+1}{(n+m)p(1-p)} \ge r/2$. Here $C>0$ is the constant given by \cref{lem:binom_shift}.
The same lower bound %$\Omega(r)$ 
$r/2$ in fact holds for any $\PP[X=Y+j]$ with %$j=(n-m)p\pm C$
$j=(n-m)p\pm \frac{\sigma}{4C}$. 
Hence,  $r=O\big(\frac{1}{\sigma}\big)=O\big(\frac{1}{\sqrt{np}}\big)$. Therefore, $\PP[X=Y]=\Theta(r)=\Theta\big(\frac{1}{\sqrt{np}}\big)$.
%Therefore, $\PP[X=Y]=O\big(\frac{1}{\sqrt{np}}\big)$ too.

\medskip

Without loss of generality we may assume that $n\ge m$. We write
\[
\PP[X\ge Y]=\PP[X-Y\ge (n-m)p]+\PP[0\le X-Y<(n-m)p].
\]
As $\max_{j}\PP[X-Y=j]=O\big(\frac{1}{\sqrt{np}}\big)$, we get $\PP[0\le X-Y<(n-m)p]=O\big(\frac{1+(n-m)p}{\sqrt{np}}\big)$. Moreover, \cref{berry} gives $\PP[X-Y\ge (n-m)p]=\Phi(0)\pm O\big(\frac{1}{\sigma}\big)=\frac{1}{2}\pm O\big(\frac{1}{\sqrt{np}}\big)$. 
Thus, it follows that
$$\PP[X\ge Y]=\frac{1}{2}\pm O\left(\frac{1+(n-m)p}{\sqrt{np}}\right),$$
as desired.
\end{proof}

\begin{lemma} \label{zw}
Let $Z_1, Z_2, W_1,W_2$ be mutually independent random variables with nonnegative integer values. For $Z:=Z_1+Z_2$ and $W:=W_1+W_2$, 
$$
- \EE[W_2] \max_{k} \PP \left[ Z_1 -W_1  =k \right] 
\leq \PP \left[ Z-W   \geq \ell \right]   - \PP \left[ Z_1-W_1   \geq \ell \right]  
\leq \EE[Z_2] \max_{k} \PP \left[ Z_1 -W  =k \right]  
~. 
$$ 
\end{lemma} 

\begin{comment}
\begin{proof} 
Note first that 
\begin{align} \label{eq:Z-W}
\PP [ Z-W   \geq \ell ] 
&= \PP [ Z_1+ Z_2 -W  \geq \ell ]\nonumber \\
&= \PP [ Z_1 -W  \geq \ell ] 
+\sum_{k=1}^{\infty} \PP [ Z_1 -W  =\ell-k \text{ and } Z_2 \geq k  ] \nonumber \\
&= \PP [ Z_1 -W  \geq \ell ] +\sum_{k=1}^{\infty} \PP [ Z_1 -W  =\ell-k ] \, 
\PP[ Z_2 \geq k  ]~,
\end{align} 
where the last equality uses independence.
An analogous argument also gives
\begin{align*} 
\PP [ Z_1-W   \geq \ell ] 
= \PP [ Z_1 -W_1-W_2 \geq \ell ] 
=\sum_{k=1}^{{\infty}} \PP [ Z_1 -W_1  =\ell+ k -1]\, 
\PP[ W_2 \leq k -1  ]\, . 
\end{align*} 
One may rearrange the sum above as
\begin{align*}  &  \sum_{k=1}^{\infty} \PP [ Z_1 -W_1  =\ell+ k -1]\, 
\PP[ W_2 \leq k -1  ] \\
&  = \sum_{k=1}^{\infty} \PP [ Z_1 -W_1  =\ell+ k -1] - 
\sum_{k=1}^{\infty} \PP [ Z_1 -W_1  =\ell+ k -1]\, 
(1- \PP[ W_2 \leq k -1  ]) \\
 &   = \PP [ Z_1 -W_1  \geq \ell] - 
\sum_{k=1}^{\infty} \PP [ Z_1 -W_1  =\ell+ k -1]\, 
 \PP[ W_2 \geq k  ]~ .
 \end{align*}
Substituting this into~\eqref{eq:Z-W} yields
 \begin{align*}
   &\PP \left[ Z-W   \geq \ell \right]  
   - \PP \left[ Z_1-W_1   \geq \ell \right]  \\
&~~~~~~~~=\sum_{k=1}^{\infty} \Big(\PP \left[ Z_1 -W  =\ell-k \right] \cdot 
\PP\left[ Z_2 \geq k \right]  -\PP \left[ Z_1 -W_1  =\ell+ k -1\right]\cdot
\PP\left[ W_2 \geq k  \right]\Big).
\end{align*}
The desired inequalities follow from the fact  $\sum_{k=1}^{\infty} \PP[ U \geq k ] = \EE[U]$
for any non negative integer valued random variable $U$. 
\end{proof}
\end{comment}

\begin{proof}
Note first that 
\begin{align} \label{eq:Z-W}
\PP [ Z-W   \geq \ell ] -\PP [ Z_1 -W  \geq \ell ]
&= \PP \left[ Z_1+Z_2-W \geq \ell \enskip \text{and} \enskip Z_1-W<\ell \right]\nonumber \\
&= 
\sum_{k=1}^{\infty} \PP [ Z_1 -W  =\ell-k \text{ and } Z_2 \geq k  ] \nonumber \\
&= \sum_{k=1}^{\infty} \PP [ Z_1 -W  =\ell-k ] \, 
\PP[ Z_2 \geq k  ]~,
\end{align} 
where the last equality uses independence.
An analogous argument also gives 
\begin{align*}
\PP \left[ Z_1-W_1   \ge \ell \right]  
   - \PP \left[ Z_1-W   \geq \ell \right]&=\PP \left[ Z_1-W_1   \geq \ell \enskip \text{and} \enskip Z_1-W_1-W_2<\ell \right]\\  
&= \sum_{k=1}^{\infty} \PP [ Z_1 -W_1  =\ell+ k -1 \enskip \text{and} \enskip W_2\ge k]\\
&=\sum_{k=1}^{\infty} \PP [ Z_1 -W_1  =\ell+ k -1]\, 
 \PP[ W_2 \geq k  ]~ .
\end{align*}
Substituting this into~\eqref{eq:Z-W} yields
 \begin{align*}
   &\PP \left[ Z-W   \geq \ell \right]  
   - \PP \left[ Z_1-W_1   \geq \ell \right]  \\
&~~~~~~~~=\sum_{k=1}^{\infty} \Big(\PP \left[ Z_1 -W  =\ell-k \right] \cdot 
\PP\left[ Z_2 \geq k \right]  -\PP \left[ Z_1 -W_1  =\ell+ k -1\right]\cdot
\PP\left[ W_2 \geq k  \right]\Big).
\end{align*}
The desired inequalities follow from the fact  $\sum_{k=1}^{\infty} \PP[ U \geq k ] = \EE[U]$
for any non negative integer valued random variable $U$. 
\end{proof}

\begin{corollary} \label{cor:4rv} Let  $X', Y', X$,~and $Y$ be  
mutually independent random variables with $\Bi(n_i,p)$ 
distributions, $i=1,2,3,4$, respectively. Then 
$$ \PP [ X' - Y'  \geq  \ell ]
= \PP [ X - Y \geq  \ell ]
\pm O\left(  \frac{p \Delta}{ \sqrt{pn_0} }\right), $$ 
where $n_0 = \min n_i$ and  
$\Delta := \max  \{ |n_1 -n_3|, |n_2-n_4| \}$.  
\end{corollary} 

\begin{comment}
\begin{proof} 
As done in~\eqref{eq:P>1/4}, it is easy to prove the fact that the maximum probability mass of $U_1-U_2$ is $O(1/\sqrt{\min\{m_1,m_2\}p})$, where $U_i\sim\Bi(m_i,p)$, $i=1,2$, are independent.
If $n_1 \geq n_3$ and $n_2 \geq n_4$, then \cref{zw} together with this fact immediately gives
the desired estimate. 
For the case that $n_1 \leq n_3$ and $n_2 \leq n_4$,
 the estimate follows from \cref{zw} with $ (Z,W)= (X, Y)$. 

If  $n_1 \geq n_3$ and $n_2 \leq n_4$, we apply  Lemma \ref{zw} twice: First, $(Z_1,W_1)= (X', Y')$ and $(Z,W)= (X', Y)$ and then $(Z_1,W_1)= (X, Y)$
with the same $(Z,W)$. As both $\PP[ X' -Y' \geq \ell]$ and  $\PP[ X -Y \geq \ell]$ are 
$\PP[ X' -Y \geq \ell]\pm O\left(  \frac{p \Delta}{ \sqrt{pn_0} }\right)$, the desired estimate holds by triangle inequality. 
The remaining case $n_1 \leq n_3$ and $n_2 \geq n_4$ also follows in an analogous way.
\end{proof}
\end{comment}

\begin{proof}
As done in~\eqref{eq:P>1/4}, it is easy to prove the fact that the maximum probability mass of $U_1-U_2$ is $O(1/\sqrt{\min\{m_1,m_2\}p})$, where $U_i\sim\Bi(m_i,p)$, $i=1,2$, are independent.

Observe that $\PP [ X' - Y'  \geq  \ell ]$ is non-decreasing in $n_1$ and non-increasing in $n_2$. So we can assume w.l.o.g. that $n_1 \geq n_3$ and $n_2 \leq n_4$. We apply  Lemma \ref{zw} twice: First, $(Z_1,W_1)= (X', Y')$ and $(Z,W)= (X', Y)$ and then $(Z_1,W_1)= (X, Y)$
with the same $(Z,W)$. As both $\PP[ X' -Y' \geq \ell]$ and  $\PP[ X -Y \geq \ell]$ are 
$\PP[ X' -Y \geq \ell]\pm O\left(  \frac{p \Delta}{ \sqrt{pn_0} }\right)$, the desired estimate holds by triangle inequality. 
\end{proof}

\section{Morning and evening on Day 0 and the next two days}\label{sec:day0}

In what follows, $s_0$ always denotes the uniform random $\pm 1$-assignment. That is, we sample each $s_0(v)$ uniformly at random from $\pm 1$ and $s_0(v)$, $v\in [n]$, are mutually independent.
Our starting point is to observe that the random initial opinion $s_0$ makes a shift of magnitude $\sqrt{n}$ with high probability.
This is in fact a standard anti-concentration result also given in~\cite[Lemma~3.1]{FKM20}, but we give a proof for completeness. 
\begin{lemma}\label{lem:deviation}
For $\varepsilon>0$, there is $c>0$ such that
$\PP\big[|\sum_v s_0(v)|\geq 2c\sqrt{n}\big]\geq 1-\varepsilon$.
\end{lemma}
\begin{proof}
The Berry--Esseen bound, \cref{berry}, gives $\PP\big[\sum_v s_0(v)\leq x\sqrt{n}/2\big] = \Phi(x) \pm O\left(1/\sqrt{n}\right)$. Choosing $x<0$ such that $\Phi(x)=1/2-\varepsilon/3$ gives, with $c:=-x/4$,
\begin{align*}
   \PP\Big[\sum_v s_0(v)\leq -2c\sqrt{n}\Big] \geq \Phi(x) + O\left(1/\sqrt{n}\right) =1/2 -\varepsilon/3 -O\left(1/\sqrt{n}\right).
\end{align*}
By symmetry, $\PP\big[|\sum_v s_0(v)|\geq 2c\sqrt{n}\big]=2\PP\big[\sum_v s_0(v)\leq -2c\sqrt{n}\big]\geq 1-2\varepsilon/3 -O\left(1/\sqrt{n}\right)$.
\end{proof}
Let $\mathcal{U}$ be the event that unanimity is achieved after a finite number of days.
Since the edges of $G=G(n,p)$ are sampled independently from the initial opinion $s_0$, $\mathcal{U}$ only depends on the value $S_0:=\sum_v s_0(v)$ (and $G=G(n,p)$) rather than what precisely $s_0$ is. 
In fact, it only depends on $|S_0|$ by symmetry. Moreover, by monotonicity, if $|S_0|$ increases, then $\mathcal{U}$ is more likely to occur.
Thus, by \cref{lem:deviation}, for $\varepsilon>0$ there exists $c>0$ such that
\begin{align}\label{eq:s0}
    \PP[\mathcal{U}]\geq 
    \PP\big[\mathcal{U}\big||S_0|\geq 2c\sqrt{n}\big]\cdot \PP\big[|S_0|\geq 2c\sqrt{n}\big]
    \geq \PP\big[\mathcal{U}\big||S_0|=2c\sqrt{n}\big]-\varepsilon,
\end{align}
where the constant $c$ is chosen to guarantee that $c\sqrt{n}$ is an integer. Hence, this ``constant" $c$ may slightly vary depending on $n$, although within the range of $\pm 1$. For brevity, we assume that $c$ is a constant and $c\sqrt{n}$ is an integer throughout this section.

\medskip

The conditional probability space given $|S_0|=2c\sqrt{n}$ can be interpreted by ``splitting" the initial assignment into two steps, namely \emph{morning} and \emph{evening} on Day 0. 
In the morning, we choose $\lceil\frac{n}{2}\rceil$ vertices $v$ to assign $+1$ and put $-1$ to the remaining $\lfloor\frac{n}{2}\rfloor$ vertices. That is, $r_0$ defined in the introduction.
We then turn signs of randomly chosen $c\sqrt{n}$ vertices $v$ with $r_0(v)=-1$, which we call \emph{swing} vertices, from $-1$ to $1$ to obtain a new $\pm 1$-assignment $\tilde{s}_0$.
To distinguish $r_0$ and $\tilde{s}_0$ from the initial opinion $s_0$, we call $r_0$ and $\tilde{s}_0$ the \emph{morning opinion} and the \emph{evening opinion}, respectively.
We also denote by $\tilde{s}_t$, $t>0$, the Day $t$ opinion resulting from majority dynamics starting with $\tilde{s}_0$.
%Once we expose $G=G(n,p)$ and $r_0$, this is a deterministic process.
Then, by \eqref{eq:s0},
\begin{align*}
    \PP\big[\mathcal{U}\big||S_0|=2c\sqrt{n}\big]
    =\PP\big[\mathcal{U}\big|s_0=\tilde{s}_0\big]
\end{align*}
for each fixed instance of $\tilde{s}_0$.
Therefore, the following main result implies~\cref{thm:main}.
Note that~$\tilde{s}_0$ depends on the constant $c>0$.
\begin{theorem}\label{thm:main_reduced}
For $\varepsilon>0$ and $\lambda>0$ there exist $c,\lambda'>0$ such that
$\PP\left[\mathcal{U}\big|s_0=\tilde{s}_0\right]\geq 1-\varepsilon$
whenever $\lambda' n^{-3/5}\log n \leq p\leq \lambda n^{-1/2}$.
\end{theorem}
To summarize, there are three types of random instances:
\begin{enumerate}[(1)]
    \item The edges of $G=G(n,p)$;
    \item The morning opinion $r_0$ chosen uniformly at random among those with exactly $\lceil n/2\rceil$ 1's;
    \item The $c\sqrt{n}$ swing vertices chosen uniformly at random from $\lfloor n/2\rfloor$ vertices $v$ with $r_0(v)=-1$.
\end{enumerate}
The edges of $G=G(n,p)$ appear independently from (2) and (3). Note that (3) is \emph{not} independent from (2), as we turn the signs of those vertices $v$ with $r_0(v)=-1$.
The distribution $\tilde{s}_0$ depends on both (2) and (3). We may also say that $\tilde{s}_0$ is obtained by ``changing" $r_0$ according to~(3).
The independence allows us to analyze probability while swapping the order of the random instances. 
For example, exposing the events in the order (1), (2) and (3) is the same as exposing some edges in (1) first, (2) and (3) second and then exposing the rest of the edges.

\medskip

Our plan is to compare the two parallel consequences of majority dynamics with the morning opinion $r_0$ and the evening opinion~$\tilde{s}_0$, respectively.
As sketched roughly in the introduction, if $\sum_{w \in N(v)} r_1(w)$ is ``almost-positive", then the vertex $v$ is highly likely to satisfy $\tilde{s}_2(v)=+1$. 
That is, such a vertex $v$ has ``many" neighbors that change their signs on Day 1 by the effect of swing neighbors and hence, $\tilde{s}_2(v)=+1$.  

As the number of such almost-positive vertices $v$ is slightly larger than $n/2$ by~\cref{lem:day0_main}, $\sum_{v} \tilde{s}_2(v)$ evaluates to a non-negligible positive value.
This is formalized by~\cref{lem:day2_bias} below. 
In what follows in this section, $\varepsilon\in(0,1)$ and $\lambda>0$ are fixed constants and we assume $G=G(n,p)$ with $\lambda' n^{-3/5}\log n\leq p\leq \lambda n^{-1/2}$,
where $\lambda'$ will be suitably chosen in the proofs.

\begin{lemma}\label{lem:day2_bias}
For each $c>0$, there exists $\alpha>0$ such that $\sum_{v\in V(G)} \tilde{s}_2(v)\geq \alpha pn^{3/2}$ w.h.p.
\end{lemma}
\begin{proof}
By exposing all the morning opinions $r_0(v)$, we may assume that $r_0$ is fixed.
We say that a vertex $w$ is \emph{unstable} if $\sum_{u\in N(w)}r_0(u)=0$. 
That is, a single swing neighbor is enough to ``change" the value of $r_1(w)$.
Given $v\in V(G)$, let $N_{-}(v)$ and $N_+(v)$ be the set of neighbors $u$ of $v$ with $r_0(u)=-1$ and $r_0(u)=+1$, respectively.
Let $B_v$ be the event that the number of unstable vertices in $N_-(v)$ with a swing neighbor is at most $\delta p^{3/2}n$, i.e., 
\[
\# \Big\{w\in N_-(v) \Big| \text{$w$ is unstable and has a swing neighbor}  \Big\} \leq \delta p^{3/2}n,
\]
where $\delta>0$ will be chosen later.
We claim that 
$\PP[B_v] \le 1/n^2$.  %$p \sqrt{n}$.
First, we expose the edges incident to $v$. Once all the neighbours of $v$ are revealed, we also expose the edges in $N(v)$. 
Then by the Chernoff bound (or \cref{lem:whp}), there exists a constant $C>0$ such that with probability at least $1-O(n^{-3})$ we have
\begin{enumerate}[(i)]
    \item both $|N_{-}(v)|$ and $|N_+(v)|$ are between $pn/2\pm n^{1/3}$.\label{it:np}
    \item each $w\in N(v)$ has at most $C\log n$ neighbors in $N(v)$.\label{it:codeg}
\end{enumerate}
We condition on the above events. 
For $w\in N(v)$, suppose that we expose all the edges incident to~$w$. Let %$a(w):=\sum_{u\in N(w)}r_0(u)$
$a(w):=\sum_{u\in N(w)\cap N(v)}r_0(u)$.
Then $w$ is unstable with probability $\PP[Y_1-Y_2+a(w)+r_0(v)= 0]$, where $Y_i\sim\Bi(n_i,p)$ and $n_1$ and~$n_2$ are vertices out of $\{v\}\cup N(v)$ with $r_0=+1$ and $-1$, respectively. 
As $|n_1-n_2|\leq \sqrt{n\log n}$, \cref{lem:binom_shift,lem:Binomial=0} give that~$w$ is unstable with probability $\Theta\big(\frac{1}{\sqrt{pn}}\big)$.
Now sample the swing vertices. Then
\begin{align*}
\PP\big[\text{$w$ has a swing neighbor}\big|\text{$w$ is unstable}\big]&=1-\binom{\lfloor n/2\rfloor-|N_{-}(w)|}{c\sqrt{n}}\bigg/ \binom{\lfloor n/2\rfloor}{c\sqrt{n}}\\
& \ge 1-\left( \frac{n/2-pn/3}{n/2}\right)^{c\sqrt{n}}\\
&\ge 1-e^{-(2c/3)p\sqrt{n}} \\
&\ge \min\left\{(c/3)p\sqrt{n},1/4\right\} \ge c'p\sqrt{n},
\end{align*}
where the first inequality follows from the fact that $|N_{-}(w)| \ge pn/3$ and $\binom{x}{k}\big/ \binom{y}{k} \le (x/y)^k$ for $k\le x \le y$ and the last uses the assumption $p\le \frac{\lambda}{\sqrt{n}}$ to obtain a constant $c'>0$.
Let $X_w$ be the indicator variable of the event that $w$ is unstable and has a swing neighbor. Then 
\[
\EE[X_w ] \ge c'p\sqrt{n}\cdot
\Theta\left(\frac{1}{\sqrt{pn}}\right)=\xi \sqrt{p}
\]
for some $\xi>0$. 
Moreover, $X_w$, $w\in N(v)$, are mutually independent given the edges in $v\cup N(v)$ are fixed. Indeed, suppose we expose all the swing vertices first and then expose the edges incident to each $w\in N(v)$ that are not contained in $N(v)$. Since each edge appears independently at random and is also independent from the choice of $r_0$ and the swing vertices, $X_w$'s are independent too. 
Let %$X\sim B(pn/2,\xi\sqrt{p})$
$X\sim B(pn/3,\xi\sqrt{p})$.
Then %$\sum_{w\in N(v)} X_w$ 
$\sum_{w\in N_{-}(v)} X_w$ stochastically dominates $X$, i.e., %$\PP\big[\sum_{w\in N(v)}X_w\leq x\big]\leq \PP[X\leq x]$ 
$\PP\big[\sum_{w\in N_{-}(v)}X_w\leq x\big]\leq \PP[X\leq x]$
for each $x\in\mathbb{R}$, since %$\sum_{w\in N(v)} X_w$ 
$\sum_{w\in N_{-}(v)} X_w$
is the sum of at least %$pn/2$ 
$pn/3$ independent Bernoulli variables with one-probability at least $\xi\sqrt{p}$. 
Then, by choosing $\delta=\xi/2$, we conclude that
\begin{align*}
\PP[B_v]&\leq \PP\left[\sum_{w\in N(v)}X_w\leq \delta p^{3/2}n\right] 
\le  O(n^{-3})+\PP[X\leq \delta p^{3/2}n] \\
         &\le O(n^{-3})+ e^{-\xi p^{3/2}n/4} \le 1/n^2.
\end{align*}
Indeed, the second inequality follows from conditioning on each edge instance on $v\cup N(v)$ that satisfies \ref{it:np} and \ref{it:codeg}. 
Then the Chernoff bound proves the next inequality.

\medskip

%\begin{proof}[Proof of \cref{lem:day2_bias}]
By the claim, with probability at least $1-O(1/n)$, no $B_v$ occurs. That is, for $c'=\xi/2$,
\begin{align}\label{eq:nobad_v}
    \# \left\{w\in N_-(v) ~\Bigg|~\sum_{u\in N(w)} r_0(u)=0, \enskip \text{$w$ has a swing neighbor}  \right\} \ge c' p^{3/2}n
\end{align}
holds for every $v\in V(G)$.
\cref{lem:day0_main} with the choice $\gamma=c'/2$ then implies that w.h.p.~there are at least $\frac{n}{2} + \alpha pn^{3/2}$ vertices~$v$ that satisfies both~\eqref{eq:nobad_v} and
\begin{align}\label{eq:star}
\sum_{w \in N(v)} r_1(w) > - \frac{c'}{2}p^{3/2}n.
\end{align}
For these vertices $v$, $\tilde{s}_2(v)=+1$, as all $w\in N_-(v)$ that is unstable and has a swing neighbor must turn to $\tilde{s}_1(w)=+1$. 
Thus, $\sum_{v}\tilde{s}_2(v) \ge \alpha pn^{3/2}$ w.h.p.
\end{proof}

\begin{remark}
The heuristic introduced in~\cite{BCOTT16} to support~\cref{conj:main} roughly predicts that the bias $|\sum_{v} s_t(v)|$ expands by a factor of $\sqrt{np}$ at each step. As $|\sum_{v} s_0(v)|=\Omega(\sqrt{n})$ with probability $1-\varepsilon$ as shown in~\cref{lem:deviation}, $|\sum_{v} s_2(v)|$ should be $\Omega(pn^{3/2})$ according to the prediction. 
This is precisely what \cref{lem:day2_bias} obtains and hence, we have just verified that the heuristic works up to Day 2 if $\lambda' n^{-3/5}\log n\leq p\leq \lambda n^{-1/2}$.
\end{remark}

\begin{proof}[Proof of~\cref{lem:day0_main}]
For a vertex $v$, let $A_v$ denote the event that $v$ is $\gamma$-almost-positive.
The plan is to use the second moment method by giving an upper bound for $\PP[A_u\cap A_v]$ and an lower bound for $\PP[A_u]$ and $\PP[A_v]$ for each pair of vertices $u$ and $v$.
The two vertices $u$ and $v$ will be fixed until these computations are carried out.
%We temporarily assume that the morning opinions $r_0$ of both $u$ and $v$ are $+1$.
%The other cases will be considered after careful computations on this case are carried out.
%The cases for the other values of $r_0(v)$ and $r_0(v)$ can be analyzed similarly and we omit the proof.

\medskip

We condition on the following high probability events.
In fact, the events hold with probability $1-O(n^{-2})$.
The constant $C>0$ below is taken large enough to apply~\cref{lem:whp} repeatedly.
\begin{enumerate}[(i)]
    \item\label{it:N(u)} First expose all the edges incident to $u$ and $v$. Then both $\deg(u)$ and $\deg(v)$ are in the interval $[np - C\sqrt{np}\log n, np + C\sqrt{np}\log n]$.
    \item\label{it:N(u,v)} The number of vertices in $N(u) \cap N(v)$ is at most $C\log n$. 
    \item\label{it:+-1} Expose $r_0$ in $\Gamma := (N(u) \cup N(v)) \setminus \{u,v\}$. The difference between the number of $\pm 1$'s in $U := N(u) \setminus (N(v) \cup \{v\})$ and in $V := N(v) \setminus (N(u) \cup \{u\})$ in the morning is at most $C\sqrt{np}\log n$. 
    \item\label{it:m1m2} Let $U_+$ and $U_-$ be the set of vertices in $U$ with the morning opinion $+1$ and $-1$, respectively, and let $m_1:=|U_+|$ and $m_2:=|U_-|$. Then both $m_1$ and $m_2$ are $pn/2\pm C\sqrt{np}\log n$.

    \item\label{it:e[N(u)]} Expose the edges inside $\Gamma$. The number of edges in each of $N(u)$ and $N(v)$ is at most $2n^2p^3$.%, \log n\}$.
    \item\label{it:a(w)}  For $w \in \Gamma$, let $a(w)$ be the sum $\sum_{x \in N(w) \cap \Gamma} r_0(x)$. Then $|a(w)| \le C\log n$ and moreover, $|\sum_{w \in \Gamma} a(w)| \le Cnp^{3/2}\log n$.
    %Condition on the fixed value of $a(w)$ 
\end{enumerate}
Indeed, \ref{it:N(u)}--\ref{it:e[N(u)]} are standard applications of the Chernoff bound and \cref{lem:whp}.
%In particular, the logarithm in~\ref{it:N(u,v)} is included for the case $np^2=O(1)$, treated by~\cref{lem:deviation}(ii).
It hence remains to check \ref{it:a(w)}. Let $\Gamma_+$ and $\Gamma_-$ be the vertices in $\Gamma$ with $r_0=+1$ and $-1$, respectively. Indeed, $|\Gamma_+|$ and $|\Gamma_-|$ are $(1+o(1))np$.
Given all the conditions (i)--(v), each $a(w)$, $w\in \Gamma_+$ is identically distributed with $X_w-Y_w$, where $X_w\sim \Bi(|\Gamma_+|-1,p)$ and $Y_w\sim \Bi(|\Gamma_-|,p)$ are independent.
If $p^2n\geq 1$, then by~\cref{lem:whp}(i), $X_w=p|\Gamma_+|\pm O(\sqrt{p|\Gamma_+|}\log|\Gamma_+|)=p^2n\pm O(p\sqrt{n}\log n)$ with probability $1-1/n^2$ and use the fact $p\sqrt{n}\leq \lambda$. Otherwise, we use \cref{lem:whp}(ii).  
The same bound also holds for $Y_w$, which proves the estimate for $|a(w)|$. The proof for the case $w\in \Gamma_-$ is almost identical.

By double counting, $\sum_{w\in \Gamma} a(w)=\sum_{ww'\in E(G[\Gamma])}(r_0(w)+r_0(w'))=2\big(e(G[\Gamma_+])-e(G[\Gamma_-])\big)$. 
This is identically distributed with $X-Y$, where $X\sim \Bi(\binom{|\Gamma_+|}{2},p)$ and $Y\sim \Bi(\binom{|\Gamma_-|}{2},p)$ are independent.
Again by~\cref{lem:whp}, we have the estimate $\frac{1}{2}n^2p^3\pm O(np^{3/2}\log n)$ for both $X$ and $Y$, which completes the proof of~\ref{it:a(w)}.

\medskip

Let $G_{uv}$ be the subgraph of $G(n,p)$ induced on $\Gamma\cup\{u,v\}$. What we have exposed so far in $G(n,p)$ precisely determines what $G_{uv}$ is.
Denote by $\mathcal{E}_{uv}$ the high probability event that all the conditions (i)--(vi) hold.
In other words,
$\mathcal{E}_{uv}$ is the collection of the pairs $(G_{uv},r_0|_{\Gamma})$ of graph instances $G_{uv}$ and values of $r_0$ in $\Gamma$ that satisfy (i)--(vi).

%With the remaining $\pm 1$'s, assign $\pm 1$ arbitrarily the vertices outside of $\{u,v\} \cup N(u) \cup N(v)$, 
Now expose $r_0$ for the remaining vertices in $V(G)\setminus\Gamma$.% Throughout the proof, we fix the choice of~$r_0$.
~We first analyze the case $r_0(u)=r_0(v)=+1$.
Let $n_1$ and $n_2$ denote the numbers of $\pm 1$'s outside $\{u,v\} \cup\Gamma$. That is, $n_1=\lceil n/2\rceil- |\Gamma_+|-2$ and $n_2=\lfloor n/2\rfloor - |\Gamma_-|$.
By~\ref{it:N(u)}, \ref{it:+-1} and \ref{it:m1m2}, both $n_1$ and $n_2$ lies between $\frac{n}{2} - np - C\sqrt{np}\log n$ and $ \frac{n}{2} - np + C\sqrt{np}\log n$. In particular, $|n_1-n_2|\leq 2C\sqrt{np}\log n$. 

For simplicity, in the proofs of~Claims~\ref{claim:1st_moment} and \ref{claim:2nd_moment}, we omit the notation that indicates conditioning on fixed $(G_{uv},r_0)$ such that $(G_{uv},r_0|_{\Gamma})\in \mathcal{E}_{uv}$ and $r_0(u)=r_0(v)=+1$.
In particular, the mean and the variance throughout Claims~\ref{claim:1st_moment} and \ref{claim:2nd_moment} are functions of $G_{uv}$ and $r_0$.

\begin{claim}\label{claim:1st_moment}
$\big|\EE\big[\sum_{w\in U} r_1(w)\big]\big| = O(\sqrt{np})$.
\end{claim}

\begin{proof}[Proof of the claim]
There is subtle asymmetry between $U_+$ and $U_-$: the vertices $w \in U_+$ turns to $+1$ after Day~1 if $\sum_{x \in N(w)} r_0(x) \ge 0$, whereas $w \in U_-$ turns to $+1$ after Day 1 if $\sum_{x \in N(w)} r_0(x) \ge 1$. 
The random variable $\sum_{w \in U_+} r_1(w)$ is identically distributed with the random variable $\sum_{w \in U_+} X_w$, where $X_w$'s are independently distributed as follows: $X_w$ takes $+1$ with probability $\PP[Y_1 + 1 + a(w) \ge  Y_2]$ and $-1$ otherwise, where $Y_i\sim \Bin(n_i,p)$ are independent binomial random variables. 
Analogously, for $w\in U_-$, $X_w$ takes $+1$ with probability $\PP[Y_1+a(w)\geq Y_2]$ and $-1$ otherwise. 

We estimate $\PP[Y_1 + a \ge  Y_2]$ for integers $a$ such that $|a|=O(\log n)$. Observe first that
\begin{align*}
    \PP[Y_1 + a \ge  Y_2] &= %\sum_{k=-a}^{\infty} \PP[Z_1=Z_2+k] = 
    \PP[Y_1\ge Y_2] +\sum_{j=1}^a\PP[Y_1+j=Y_2] ~~\text{ if $a>0$ and }\\
    \PP[Y_1+a \ge  Y_2] &= \PP[Y_1\ge Y_2] -\sum_{j=0}^{-a-1}\PP[Y_1-j=Y_2] ~~\text{ if $a<0$. }
\end{align*}
%Otherwise, if $a<0$, then
%$     \PP[Z_1 + a \ge  Z_2] = \PP[Z_1\ge Z_2] -\sum_{j=0}^{-a-1}\PP[Z_1+j=Z_2]$.
\cref{lem:binom_shift} then allows us to approximate $\PP[Y_1 + a \ge  Y_2]$ by $\PP[Y_1\ge Y_2] + a\PP[Y_1=Y_2]$.
Namely, if $0\le j\le C \log n$, then $|\PP[Y_1+j=Y_2]-\PP[Y_1=Y_2]| =O\big(\frac{\log n}{np}\big)$ and hence, for $|a| =O(\log n)$,
\begin{align}\label{eq:Z+a}
    \PP[Y_1 + a \ge  Y_2] = \PP[Y_1\ge Y_2] + a \left(\PP[Y_1=Y_2] \pm O\left(\frac{\log n}{np}\right)\right).
\end{align}
Almost the same argument also proves
\begin{align}\label{eq:Z+a+1}
    \PP[Y_1 + a+1 \ge  Y_2] = \PP[Y_1+1\ge Y_2] + a\left( \PP[Y_1=Y_2] \pm O\left(\frac{\log n}{np}\right)\right).
\end{align}
For brevity, let $q = \PP[Y_1 =Y_2]$ and $p_k = \PP[Y_1 + k \ge Y_2]$. By~\cref{lem:Binomial=0}, $q=\Theta\big(\frac{1}{\sqrt{np}}\big)$.
%, which is asymptotically larger than $\log n/(np)$ if $p\gg \log^2 n/n$; we may thus write $q\pm \log n/n = (1+o(1))q$.
Let $\mu_{w}$ be the expectation of the random variable $r_1(w)$ conditioned on (i)--(vi) and $r_0(u)=r_0(v)=+1$.
Then by using \eqref{eq:Z+a+1} and~\eqref{eq:Z+a} for $w\in U_+$ and $w\in U_-$, respectively,
\begin{align*}
    \mu_w &= 2p_{a(w)+1}-1 = 2p_1-1 +2a(w)\left(q\pm O\left(\frac{\log n}{np}\right)\right)~\text{ for }w\in U_+~\text{ and}\\
    \mu_w &= 2p_{a(w)}-1 = 2p_0-1 +2a(w)\left(q\pm O\left(\frac{\log n}{np}\right)\right)~\text{ for }w\in U_-.
\end{align*}
Let $\mu_{\pp}:=\sum_{w\in U}\mu_w$ to indicate that it is conditioned on $r_0(v)=r_0(u)=+1$. Then
\begin{align}\label{eq:mu}
    &\nonumber\mu_{\pp} = m_1(2p_1 - 1) + m_2(2p_0 - 1) + 2\sum_{w \in U} a(w)\left(q\pm O\bigg(\frac{\log n}{np}\bigg)\right)\\
    &= np(p_0+p_1 -1) + \bigg(p_1-\frac{1}{2}\bigg)(2m_1-np)+\bigg(p_0-\frac{1}{2}\bigg)(2m_2-np)
    + 2\sum_{w \in U} a(w)\left(q\pm O\bigg(\frac{\log n}{np}\bigg)\right) \nonumber \\
    &= np(p_0 + p_1 - 1) \pm  O\left(\log^2 n\right).
    %+ p\sqrt{n} \log n\right).
\end{align}
Indeed, $m_1$ and $m_2$ are $np/2 \pm O(\sqrt{np}\log n)$ by~\ref{it:m1m2} and  both $p_0$ and $p_1$ are $\frac{1}{2}\pm O\big(\frac{1+|n_1-n_2|p}{\sqrt{np}}\big)$. As $|n_1-n_2|=O(\sqrt{np}\log n)$, $p_0$ and $p_1$ are $\frac{1}{2} \pm O\big(\frac{1}{\sqrt{np}}\big)$ by ~\cref{lem:Binomial=0}.
Thus, both $(p_1-1/2)(2m_1-np)$ and $(p_0-1/2)(2m_2-np)$ are $O(\log n)$.
We also use~\ref{it:a(w)} and the fact $q=\Theta\big(\frac{1}{\sqrt{np}}\big)$ to obtain the bound $q\left|\sum_{w\in U}a(w)\right|=O(p\sqrt{n}\log n)$. Moreover, $|a(w)|\leq\log n$ by~\ref{it:a(w)}, so $\frac{\log n}{np}\sum_{w\in U}|a(w)|=O(\log^2 n)$.
Overall, $|\mu_{\pp}|=O(\sqrt{np})$.
\end{proof}

\begin{claim}\label{claim:2nd_moment}
$\mathrm{Var}\big(\sum_{w\in U} r_1(w)\big) = np\pm \sqrt{np}\log n$.
\end{claim}
\begin{proof}[Proof of the claim]
%We now estimate the variance of $r_1(w)$. 
For $w\in U_+$,
\begin{align*}
    \mathrm{Var}(r_1(w)) &= 1-\mu_w^2 = 
    1 - \big(2p_1 -1 + 2a(w)(1+o(1))q \big)^2 \\
    &= 1 - (2p_1 - 1)^2 + 4a(w)(2p_1 - 1)(1+o(1))q +4a(w)^2(1+o(1))q^2\\
    &= 1-(2p_1-1)^2 \pm O\left(\frac{\log^2 n}{np}\right)
    %&= 1-(2p_1-1)^2 +4a(w)(2p_1-1)(1+o(1))q,
\end{align*}
where the last equality follows from~\ref{it:a(w)}, $p_1=1/2\pm O\big(\frac{1}{\sqrt{np}}\big)$ and $q=\Theta\big(\frac{1}{\sqrt{np}}\big)$. % with $p\gg n^{-2/3}$. 
For $w\in U_-$, an analogous bound holds with $p_0$ instead of $p_1$.
Let $\sigma_{\pp}^2$ be the variance of $\sum_{w\in U}r_1(w)$, conditioned on the fixed $G_{uv}$ and $r_0$. By
\begin{align*}
    m_1(1-(2p_1-1)^2)  %\frac{np}{2}\left(1-(2p_1-1)^2\right)+ \left(m_1-\frac{np}{2}\right)\left(1-(2p_1-1)^2\right)
    =\frac{np}{2}\left(1-(2p_1-1)^2\right) \pm \sqrt{np}\log n
\end{align*}
and a similar bound for $m_2(1-(2p_0-1)^2)$,
we obtain
\begin{align*}
    %\mathrm{Var}\left(\sum_{w \in U} r_1(w)\right) 
    \sigma_{\pp}^2 &= m_1(1 - (2p_1 - 1)^2) + m_2(1 - (2p_0 - 1)^2) \pm O(\log^2 n)\\
    &= \frac{np}{2}\left(2 - (2p_1 - 1)^2 - (2p_0 - 1)^2\right) \pm O\left(\sqrt{np}\log n\right)\\
    &=np\pm O(\sqrt{np}\log n),
\end{align*}
as $\log^2 n\ll \sqrt{np}\log n$ as $p\gg \log^2 n/n$.
Thus, $\sigma_{\pp}=\sqrt{np}\pm O(\log n)$.
\end{proof}
We now turn to analyze the other cases with different signs of $r_0$ on $u$ and $v$.
Recall that the high probability event $\mathcal{E}_{uv}$ consists of pairs $(G_{uv},r_0|_{\Gamma})$ of the graph $G_{uv}$ on $\Gamma\cup\{u,v\}$ and $r_0$ restricted on $\Gamma$ that satisfy (i)--(vi). 
For simplicity, we write $G_{uv}^*$ for the pair $(G_{uv},r_0|_{\Gamma})$.

For fixed $G_{uv}$ and $r_0$ such that $G_{uv}^*\in\mathcal{E}_{uv}$ and $r_0(u)=r_0(v)=+1$, 
suppose that only $r_0(v)$ changes from $+1$ to $-1$ while everything else remains the same.
Then, in the proofs of Claims~\ref{claim:1st_moment} and~\ref{claim:2nd_moment}, $n_1$ and $n_2$ are very slightly changed: $n_1$ increases by $1$ and $n_2$ decreases by $1$. 
However, the arguments throughout the proofs remain exactly the same. 
The conditional mean, denoted by $\mu_{\pmi}$, in this case can differ from $\mu_{\pp}$ only very slightly.
The only difference is the values of $n_1$ and~$n_2$, which makes $p_0$, $p_1$ and $q$ differ by $O(p)$ by \cref{cor:4rv}.
Including this error term in~\eqref{eq:mu} gives $\mu_{\pp}=\mu_{\pmi}\pm O(p^2n + \log^2 n)$.
As $p^2n\ll \log^2 n$, we have $\mu_{\pp}=\mu_{\pmi}\pm O(\log^2 n)$.

If $r_0(u)=-1$ and $r_0(v)=+1$ in the same setting, the conditional expectation and the conditional variance, denoted by $\mu_{\mpl}$ and $\sigma_{\mpl}$, respectively, are estimated by the same method with slightly different parameters.
More precisely, $n_1'=\lceil n/2 \rceil-|\Gamma^+|-1$ and $n_2'=\lceil n/2 \rceil-|\Gamma^-|-1$.
Let $p_0' =\PP[Y_1'-1\geq Y_2']$ and $p_1'=\PP[Y_1'-2\geq Y_1']$, where $Y_i'\sim\Bi(n_i,p)$, $i=1,2$.
Similarly to~\eqref{eq:mu}, one then obtains the bound
\begin{align*}
    \mu_{\mpl} = np(p_0' + p_1' - 1) \pm  O\left(\log^2 n\right).
\end{align*}
In particular, $\mu_{\mpl}=O(\sqrt{np})$ and $\sigma_{\mpl}=\sqrt{np}\pm O(\log n)$.
Indeed, these bounds remain the same if $r_0(u)=r_0(v)=-1$ and the only difference from the case $r_0(u)=-1$ and $r_0(v)=+1$ is the values of $n_1$ and $n_2$,
so $\mu_{\mpl}=\mu_{\mm}\pm O(\log^2 n)$.
Let $\mu_-:=\frac{1}{2}(\mu_{\mpl}+\mu_{\mm})$. 
Then we also have that $\mu_{\mpl}$ and $\mu_{\mm}$ are $\mu_-\pm O( \log^2 n)$.
Overall, the bound $np\pm O(\sqrt{np}\log n)$ is universal for the variance obtained in all the four cases.
To summarize, we so far have
\begin{align}\label{eq:mu_sigma}
    |\mu_{\pmi} -\mu_{\pp}|=O(\log^2 n), ~|\mu_{\mpl}-\mu_{\mm}|=O(\log^2 n) ~\text{ and }~ \sigma = \sqrt{np}\pm O(\log n),
\end{align}
where $\sigma$ can be $\sigma_{\pp},\sigma_{\pmi},\sigma_{\mpl}$ or $\sigma_{\mm}$.
Despite the estimates above, we only obtained the bound $O(\sqrt{np})$ for $|\mu_+ +\mu_-|$ by \cref{claim:1st_moment}. This is not enough for our purpose, which motivates the following claim.

\begin{claim}\label{claim:mu+-}
% $\big|\EE\big[\sum_{w\in U}r_1(w)\mathbf{1}_{\mathcal{E}_{uv}}\big]\big|=O(p^2n).$
For every $G_{uv}$ and $r_0$ with $G_{uv}^*\in\mathcal{E}_{uv}$, $\big|\mu_+ +\mu_-\big|=O(\log^2n)$.
\end{claim}
\begin{proof}[Proof of the claim]
Let $Z$ and $Z'$ be i.i.d.~variables with the distribution $\Bi(\lceil n/2 -np \rceil,p)$. %$\Bi(\lceil n/2 -2np \rceil,p)$.
Note first that $n_1$ and $n_2$ in each of the four cases depending on the signs of $r_0(u)$ and $r_0(v)$ vary from $\lceil n/2 -np \rceil$ %$\lceil n/2 -2np \rceil$ 
by at most $O(\sqrt{np}\log n)$ by~\ref{it:+-1}. 
Let $(p_0,p_1)$ and $(p_0',p_1')$ be as defined in the cases $r_0(u)=r_0(v)=+1$ and $r_0(u)=-1,r_0(v)=+1$ above.
\cref{cor:4rv} then yields 
\begin{align*}
    &p_0 = \PP[Z\geq Z'] \pm O(p\log n),
    ~~p_1 = \PP[Z+1\geq Z']\pm O(p\log n),\\
    &p_0' = \PP[Z'-1\geq Z] \pm O(p\log n),\text{ and }~ p_1' = \PP[Z'-2\geq Z] \pm O(p\log n).
\end{align*}
In particular, $p_0+p_0' = 1\pm O(p\log n)$ and $p_1+p_1'=1\pm O(p\log n)$. Therefore, as $np^2\log n\ll \log^2 n$,
\begin{align*}
    \left|\mu_{\pp}+\mu_{\mpl}\right| = \left|np(p_0+p_1+p_0'+p_1'-2) \pm O\hspace{-.5mm}\left(\log^2 n\right)\right|
    = O\hspace{-.5mm}\left(\log^2 n\right).
\end{align*}
An analogous coupling argument works for $\mu_{\pmi}$ and $\mu_{\mm}$.
Hence, \begin{align*}
    |\mu_+ +\mu_-|\leq \frac{1}{2}\big(|\mu_{\pp}+\mu_{\mpl}|+|\mu_{\pmi}+\mu_{\mm}|\big) = O(\log^2n).\tag*{\qedhere}
\end{align*}
\end{proof}
Let $R_{uv}^{\pp}$, $R_{uv}^{\pmi}$, $R_{uv}^{\mpl}$ and $R_{uv}^{\mm}$ be the events that $r_0(u)$ and $r_0(v)$ take the corresponding signs, respectively.
Then the probability of each of the four events is easily computed as $1/4\pm O(p)$ given $G_{uv}^*\in\mathcal{E}_{uv}$, e.g.,
\begin{align}\label{eq:1/4}
    \PP[R_{uv}^{\pp}|G_{uv}^*]=
    \frac{\binom{n-|\Gamma|-2}{\lceil n/2\rceil-m_1-2}}{\binom{n-|\Gamma|}{\lceil n/2\rceil-m_1}} = \frac{(\lceil n/2\rceil-m_1)(\lceil n/2\rceil-m_1-1)}{(n-|\Gamma|)(n-|\Gamma|-1)} = \frac{1}{4}\pm O(p).
\end{align}
We are now ready to estimate the variance of $\sum_{u\in V(G)} \mathbf{1}_{A_u}$.
\begin{claim}\label{claim:covariance}
$\mathrm{Var}\left(\sum_{u\in V(G)} \mathbf{1}_{A_u}\right)=\sum_{u,v\in V(G)} \PP[A_u \cap A_v]-\PP[A_u]\cdot \PP[A_v]=O\left(\frac{n^{3/2}\log^2 n}{\sqrt{p}}\right)$.
\end{claim}
\begin{proof}[Proof of the claim]
Let $A_u^\pm$ be the events that $\sum_{w\in U}r_1(w)>-\gamma p^{3/2}n\pm C\log n$, respectively with the corresponding signs. 
In particular, for $G_{uv}\in \mathcal{E}_{uv}$, $A_u^+$ implies $A_u$ and $A_u$ implies $A_u^-$ by \ref{it:N(u,v)}.
The mutual independence of all $r_1(w)$, $w\in U$, given fixed $G_{uv}$ and $r_0$, allows us to apply the Berry--Esseen bound.
For each fixed $G_{uv}$ and $r_0$ such that  $r_0(u)=r_0(v)=+1$ and $G_{uv}^*\in\mathcal{E}_{uv}$,
\begin{align}\label{eq:lower_psi}
     \nonumber\PP[A_u|G_{uv},r_0] &\geq 
     \PP[A_u^+|G_{uv},r_0] \\
     &= \PP\left[\frac{\sum_{w \in U} r_1(w) - \mu_{\pp}}{\sigma_{\pp}} > \frac{- \gamma p^{3/2}n + C\log n - \mu_{\pp}}{\sigma_{\pp}}~\bigg|~G_{uv},r_0\right]\nonumber\\
     &= \Psi\left(\frac{- \gamma p^{3/2}n +C\log n - \mu_{\pp}}{\sigma_{\pp}}\right) \pm O\left(\frac{1}{\sqrt{np}}\right)\nonumber\\
     &\geq \Psi\left(\frac{- \gamma p^{3/2}n+C'\log^2 n - \mu_+}{\sigma_{\pp}}\right) - O\left(\frac{1}{\sqrt{np}}\right).
\end{align}
where $\Psi(x):=1-\Phi(x)$ as in~\cref{lem:psi} and $C'>0$ is from the estimate $|\mu_+-\mu_{\pp}|=O(\log^2n)$ by~\eqref{eq:mu_sigma}, which absorbs $C\log n$.
By using $A_u^-$, one also obtains the upper bound
\begin{align}\label{eq:upper_psi}
     \PP[A_u|G_{uv},r_0] 
     %\leq \PP[A_u^-|\mathcal{E}_{uv}\cap R_{uv}^{\pp}] 
     &\leq \PP\left[\frac{\sum_{w \in U} r_1(w) - \mu_{\pp}}{\sigma_{\pp}} > \frac{- \gamma p^{3/2}n -C \log n - \mu_{\pp}}{\sigma_{\pp}}~\bigg|~G_{uv},r_0\right]\nonumber\\
     &\leq \Psi\left(\frac{- \gamma p^{3/2}n -C'\log^2 n- \mu_{+}}{\sigma_{\pp}}\right) + O\left(\frac{1}{\sqrt{np}}\right).
\end{align}
Both bounds~\eqref{eq:lower_psi} and~\eqref{eq:upper_psi} can be written as
$\Psi(x_+)\pm O\big(\frac{\log^2 n}{\sqrt{np}}\big)$, where $x_+ = (-\gamma p^{3/2}n-\mu_+)/\sqrt{np}$.
Indeed, by using~\cref{lem:psi}(i), i.e., $|\Psi(x)-\Psi(y)| \le |x-y|$,
\begin{align*}
   \left|\Psi\left(\frac{- \gamma p^{3/2}n+ C'\log n - \mu_+}{\sigma_{\pp}}\right) - \Psi(x_+)\right|
   &\leq 
    C'\log^2 n \left(\frac{1}{\sigma_+}+\frac{1}{\sqrt{np}}\right) + \left|\gamma p^{3/2}n+\mu_+\right|\left|\frac{1}{\sigma_{\pp}}-\frac{1}{\sqrt{np}}\right|\\
    &\leq \frac{3C'\log^2 n}{\sqrt{np}} +O(\sqrt{np})\cdot\frac{|\sigma_{\pp}-\sqrt{np}|}{\sigma_{\pp}\sqrt{np}}
    =O\left(\frac{\log^2 n}{\sqrt{np}}\right),
\end{align*}
where we use the estimates $\sigma_{\pp}=\sqrt{np}\pm O(\log n)$ by Claim~\ref{claim:2nd_moment} and $\mu_+=O(\sqrt{np})$ by Claim~\ref{claim:1st_moment}.
The same bound also holds for $\Psi\left(\frac{- \gamma p^{3/2}n- C'\log^2 n - \mu_+}{\sigma_{\pp}}\right)$.
Now replace $\sigma_{\pp}$ by $\sigma_{\pmi}$ by changing $r_0(v)$ from $+1$ to $-1$, while leaving all other values of $r_0$ and $G_{uv}$ the same.
Then we again have the same bound $\PP[A_u|G_{uv},r_0] = \Psi(x_+)\pm O\left(\frac{\log^2 n}{\sqrt{np}}\right)$.
Analogously, for $r_0$ with $r_0(u)=-1$,
we obtain $\PP[A_u|G_{uv}, r_0]=\Psi(x_-)\pm O\left(\frac{\log^2 n}{\sqrt{np}}\right)$,
where $x_-=(-\gamma p^{3/2}n -\mu_-)/\sqrt{np}$,
by using $\sigma_{\mpl},\sigma_{\mm}$ and $\mu_-$.

Given fixed $r_0$ and $G_{uv}$, 
$A_u^-$ and $A_v^-$ are independent as $r_1(w)$, $w\in U\cup V$, are mutually independent. Hence, for each fixed $G_{uv}$ and $r_0$ such that $r_0(u)=r_0(v)=+1$ and $G_{uv}^*\in\mathcal{E}_{uv}$,
\begin{align*}
    \PP[A_u \cap A_v|G_{uv},r_0] 
    &\leq \PP[A_u^- \cap A_v^-|G_{uv},r_0] \\
    &=\PP[A_u^-|G_{uv},r_0]\cdot \PP[A_v^-|G_{uv},r_0]
    \leq \Psi(x_{+})^2  +O\left(\frac{\log^2 n}{\sqrt{np}}\right).
\end{align*}
Indeed, one may easily obtain corresponding upper bounds for other values of $r_0(u)$ and $r_0(v)$, e.g., $\PP[A_u \cap A_v|G_{uv},r_0]=\Psi(x_{+})\Psi(x_-)  +O\left(\frac{\log^2 n}{\sqrt{np}}\right)$ if $r_0(u)=+1$ and $r_0(v)=-1$.
Combining these bounds with the weight $1/4\pm O(p)$ in~\eqref{eq:1/4} gives
\begin{align*}
    \PP[A_u \cap A_v|G_{uv}^*] &= \EE\big[\PP[A_u \cap A_v|G_{uv},r_0]\big|G_{uv}^*\big]\\
    &\leq \frac{1}{4}\EE\hspace{-.5mm}\left[\Psi(x_{+})^2 +2\Psi(x_{+})\Psi(x_{-})+\Psi(x_{-})^2\Big|G_{uv}^*\right]+O\left(\frac{\log^2 n}{\sqrt{np}}\right)\\
   & =\frac{1}{4}\EE\hspace{-.5mm}\left[\big(\Psi(x_{+})+\Psi(x_{-})\big)^2\Big|G_{uv}^*\right]+O\left(\frac{\log^2 n}{\sqrt{np}}\right),
\end{align*}
where the $O(p)$ error term in the weight $1/4\pm O(p)$ is absorbed by $O\big(\frac{\log^2 n}{\sqrt{np}}\big)$.
Summing this bound over all $G_{uv}^*\in\mathcal{E}_{uv}$ with the corresponding probability weight that $G_{uv}^*$ appears yields
\begin{align*}
    \PP[A_u \cap A_v\cap \mathcal{E}_{uv}]\leq \frac{1}{4}\EE\hspace{-.5mm}\left[\big(\Psi(x_{+})+\Psi(x_{-})\big)^2\mathbf{1}_{\mathcal{E}_{uv}}\right]+O\left(\frac{\log^2 n}{\sqrt{np}}\right).
\end{align*}
Analogously, \eqref{eq:lower_psi} and its variants give the lower bound
\begin{align*}
    \PP[A_u\cap \mathcal{E}_{uv}]\cdot \PP[A_v\cap \mathcal{E}_{uv}]\geq 
    \frac{1}{4}\EE\hspace{-.5mm}\left[\left(\Psi(x_{+})+\Psi(x_{-})\right)^2\mathbf{1}_{\mathcal{E}_{uv}}\right]-O\left(\frac{\log^2 n}{\sqrt{np}}\right).
\end{align*}
Summing the above over distinct $u,v\in V(G)$ gives
\begin{align*}
   \sum_{u\neq v} \PP[A_u \cap A_v\cap\mathcal{E}_{uv}]-\PP[A_u\cap\mathcal{E}_{uv}]\cdot \PP[A_v\cap \mathcal{E}_{uv}] = O\left(\frac{n^{3/2}\log^2 n}{\sqrt{p}}\right).
\end{align*}
Therefore, we estimate the variance of $\sum_{v\in V(G)}\mathbf{1}_{A_u}$ as
\begin{align}\label{eq:covar}
    \nonumber\sum_{u,v\in V(G)} &\PP[A_u \cap A_v]-\PP[A_u]\cdot \PP[A_v] \le n+\sum_{u\neq v} \PP[A_u \cap A_v]-\PP[A_u]\cdot \PP[A_v]\\ \nonumber
    &\leq 
    n+\sum_{u\neq v} \left(\PP[A_u \cap A_v\cap \mathcal{E}_{uv}]+\PP[\overline{\mathcal{E}}_{uv}]\right)-%\big(1-O(n^{-2})\big)
    \PP[A_u\cap\mathcal{E}_{uv}]\cdot \PP[A_v\cap \mathcal{E}_{uv}]\\ 
    &\leq O(n)+\sum_{u\neq v} \PP[A_u \cap A_v\cap \mathcal{E}_{uv}]-\PP[A_u\cap \mathcal{E}_{uv}]\cdot \PP[A_v\cap \mathcal{E}_{uv}] = O\left(\frac{n^{3/2}\log^2 n}{\sqrt{p}}\right),
\end{align}
where the last inequality uses the bound $\PP[\overline{\mathcal{E}}_{uv}]=O(n^{-2})$. This concludes the proof of the claimed variance estimate.
\end{proof}
It remains to bound $\sum_{u} \PP[A_u]$ from below to use Chebyshev's inequality.
Note first that $r_0(u)$ takes each sign with probability $1/2\pm O(p)$ given $G_{uv}^*$, which can easily be computed by an analogous estimate to~\eqref{eq:1/4}.
Recall that, depending on the sign of $r_0(u)$, $\PP[A_u|G_{uv},r_0]$ can be estimated as either $\Psi(x_+)\pm O\left(\frac{\log^2 n}{\sqrt{np}}\right)$ or $\Psi(x_-)\pm O\left(\frac{\log^2 n}{\sqrt{np}}\right)$. 
Hence, for each $G_{uv}^*\in\mathcal{E}_{uv}$,
\begin{align}\label{eq:Au_lower}
    \PP[A_u|G_{uv}^*] = \EE\big[\PP[A_u |G_{uv},r_0]\big|G_{uv}^*\big]\geq
    \frac{1}{2}\EE\hspace{-.5mm}\left[\Psi(x_+)+\Psi(x_-)\Big|G_{uv}^*\right]- O\left(\frac{\log^2 n}{\sqrt{np}}\right).
\end{align}
For each fixed $G_{uv}$ and $r_0$ with $G_{uv}^*\in\mathcal{E}_{uv}$, both $x_+$ and~$x_-$ are $O(1)$, as $|\mu_\pm|+\gamma p^{3/2}n= O(\sqrt{np})$ by~\cref{claim:1st_moment}.
Moreover, 
\begin{align*}
    x_+ +x_- = -\frac{2}{\sqrt{np}}\left(\gamma p^{3/2}n +\mu_+ +\mu_-\right)\leq -\gamma p\sqrt{n}<0,
\end{align*}
as $|\mu_+ +\mu_-|=O(\log^2 n)\ll p^{3/2}n$ by~\cref{claim:mu+-}.
\cref{lem:psi}(ii) then gives a constant $C''>0$ such that 
\begin{align*}
    \Psi(x_+)+\Psi(x_-)\geq 1-C''(x_+ +x_-) \geq 1+C''\gamma p \sqrt{n}.
\end{align*}
Substituting this into~\eqref{eq:Au_lower} and summing over all $G_{uv}^*\in\mathcal{E}_{uv}$ gives
\begin{align*}
    \sum_{u\in V(G)\setminus\{v\}}\PP[A_u\cap\mathcal{E}_{uv}]\geq\frac{n-1}{2}\big(\Psi(x_+)+\Psi(x_-)\big)- O\left(\frac{\sqrt{n}\log^2 n}{\sqrt{p}}\right)
    \geq \frac{n}{2}+3K\gamma pn^{3/2}
\end{align*}
for a constant $K>0$, as $pn^{3/2}\gg \log^2 n\sqrt{n/p}$.
Thus,
\begin{align*}
    \sum_{u\in V(G)}\PP[A_u] \geq 
    \sum_{u\in V(G)\setminus\{v\}}\PP[A_u\cap\mathcal{E}_{uv}] 
    \geq \frac{n}{2}+2K\gamma pn^{3/2}.
\end{align*}
Finally, together with \eqref{eq:covar}, the Chebyshev inequality yields
\begin{align*}
    \PP\left[\sum_u \mathbf{1}_{A_u} \geq \frac{n}{2} + K\gamma pn^{3/2}\right] &\leq
    \PP\left[\left|\sum_{u\in V(G)} \mathbf{1}_{A_u} - \mathbb{E}\left(\sum_{u\in V(G)} \mathbf{1}_{A_u}\right)\right| \ge K\gamma pn^{3/2}\right]  \\
    &\le \frac{\mathrm{Var}\left(\sum_{u} \mathbf{1}_{A_u}\right)}{\left(K\gamma pn^{3/2}\right)^2}
    = O\left(\frac{\log^2 n}{p^{5/2}n^{3/2}}\right),
\end{align*}
where the last estimate follows from~\cref{claim:covariance}.
Finally, we have $O\left(\frac{\log^2 n}{p^{5/2}n^{3/2}}\right)=O\left(\frac{1}{\sqrt{\log n}}\right)=o(1)$, as $p\geq \lambda' n^{-3/5}\log n$.
\end{proof}

\section{After Day 2}\label{sec:day1+}

To finish the proof of~\cref{thm:main_reduced}, we use some well-known ``pseudorandom" properties of random graphs $G(n,p)$.
By~\cref{lem:whp}, w.h.p. the minimum degree of $G(n,p)$ is at least $0.9np$ whenever $p\ge  \frac{\log^2 n}{n}$.
%Second, 
We say a graph $G$ is \emph{$(p,\beta)$-jumbled} if, for any vertex subsets $U,V\subseteq V(G)$,
\begin{align*}
    \big|e(U,V)-p|U||V|\big|\leq \beta\sqrt{|U||V|}.
\end{align*}
The following is a standard fact in the theory of pseudorandomness.
\begin{lemma}[{\cite[Corollary~2.3]{KS06}}]\label{lem:sqrtnp}
For $p\leq 0.99$, $G(n,p)$ is w.h.p. $(p,\beta)$-jumbled with $\beta=O(\sqrt{np})$.
\end{lemma}

Let $P_{t}:=\{v\in[n]:s_t(v)=+1\}$ and $N_t:=\{v\in[n]:s_t(v)=-1\}$.
We use \cite[Lemma~7 and 8]{Z18} by Zehmakan. Here we give a short proof, as we need a slightly more general version.
\begin{lemma}\label{lem:Zehmakan}
Let $\delta\in (0,1)$ and let $G$ be a $(p,\beta)$-jumbled graph on $n$ vertices with minimum degree at least $\delta np$. Then
\begin{enumerate}[(i)]
    \item if $\sum_{v}s_t(v)\geq \frac{8\beta}{p\sqrt{\delta}}$ then $\sum_{v}s_{t+1}(v)\geq (1-\delta/2)n$;
    \item if $\sum_{v}s_{t}(v)\geq (1-\alpha) n $ for some $\alpha\le \delta/2$, then $\sum_{v}s_{t+1}(v)\geq \left(1-\alpha\cdot \frac{16\beta^2}{\delta^2n^2p^2}\right)n$.
\end{enumerate}
\end{lemma}
\begin{proof}
(i) Since each vertex $v\in N_{t+1}$ has at least as many neighbours in $N_t$ as in $P_{t}$, we have $e(N_{t+1},N_t)\geq e(N_{t+1},P_{t})$.
    Then by $(p,\beta)$-jumbledness,
    \begin{align*}
        p|N_{t+1}||P_{t}|-\beta\sqrt{|N_{t+1}||P_{t}|}\leq e(N_{t+1},P_t)\leq e(N_{t+1},N_t)\leq  p|N_{t+1}||N_t|+\beta\sqrt{|N_{t+1}||N_t|}.
    \end{align*}
Dividing both ends of the inequality above by $\sqrt{|N_{t+1}|}$ gives
\begin{align*}
    p\sqrt{|N_{t+1}|}\left(|P_{t}|-|N_t|\right)\leq 2\beta \left(\sqrt{|N_t|}+\sqrt{|P_{t}|}\right) \le 4\beta \sqrt{n}.
\end{align*}
Thus,
\begin{align*}
    \sqrt{|N_{t+1}|}\leq \frac{4\beta\sqrt{n}}{p\sum_v s_t(v)}
    \leq \frac{\sqrt{\delta n}}{2}.
\end{align*}
Hence, $|N_{t+1}|\leq \delta n/4$, which means $\sum s_{t+1}(v) = n-2|N_{t+1}|\geq (1-\delta/2)n$.

\medskip

\noindent (ii) Each vertex $v\in N_{t+1}$ has at least $\deg(v)/2$ neighbours in $N_t$.
As the minimum degree of $G$ is at least $\delta np$, this means 
$e(N_t,N_{t+1})\geq \frac{\delta np}{2}|N_{t+1}|$.
Combining this with $(p,\beta)$-jumbledness yields
\begin{align*}
    \frac{\delta np}{2}|N_{t+1}|\leq e(N_t,N_{t+1})\leq  p|N_t||N_{t+1}|+\beta\sqrt{|N_t||N_{t+1}|}.
\end{align*}
As $\sum_{v}s_{t}(v)\geq (1-\alpha) n$, it follows that $|N_t|\leq \alpha n/2$. Thus,
\begin{align*}
    \sqrt{|N_{t+1}|}\left(\frac{\delta np}{2}-p|N_t|\right)\leq \beta \sqrt{|N_t|}\leq \beta\sqrt{\alpha n/2}.
\end{align*}
On the other hand, by $\alpha\leq \eps/2$ and $|N_t|\leq \alpha n/2$, 
\begin{align*}
    \frac{\delta np}{2}-p|N_t|
    \geq \frac{\delta np}{4},
\end{align*}
which means $\frac{\delta np}{4}\sqrt{|N_{t+1}|}\leq \beta\sqrt{\alpha n/2}$. % i.e., 
%$|N_{t+1}|\leq \frac{32\alpha \beta^2}{np^2}$. 
Hence, $\sum_{v}s_{t+1}(v)=n-2|N_{t+1}| \geq \left(1-\alpha\cdot \frac{16\beta^2}{\delta^2 n^2p^2}\right)n$.
\end{proof}

\begin{proof}[Proof of~\cref{thm:main_reduced}]
Choose $\lambda'$ such that $p\geq \lambda'n^{-3/5}\log n$ according to~\cref{lem:day2_bias} so that $\sum_{v}\tilde{s}_{2}(v)\geq \alpha pn^{3/2}$ with probability at least $1-\varepsilon$.
On the other hand, by~\cref{lem:sqrtnp}, $G(n,p)$ is w.h.p. a $(p,\beta)$-jumbled graph with minimum degree at least $\delta np$, where $\beta=O(\sqrt{np})$ and $\delta=0.9$. 
As $p\gg n^{-2/3}$,
\begin{align*}
    \sum_{v}s_2(v)\geq \alpha pn^{3/2} \gg \frac{8\beta}{p\sqrt{\delta}} = O\left(\sqrt{\frac{n}{p}}\right)
\end{align*}
and therefore,~\cref{lem:Zehmakan}~(i) proves that $\sum_v s_3(v)\geq (1-\delta/2)n$. Then iterating \cref{lem:Zehmakan}~(ii) for $k$ times gives 
\begin{align*}
    \sum_v s_{k+3}(v)\geq \left(1-\frac{\delta}{2}\left(\frac{16\beta^2}{\delta^2 n^2p^2}\right)^k\right) n.
\end{align*}
If $k= 3$, then $\left(\frac{16\beta^2}{\delta^2 n^2p^2}\right)^k<1/n$ and thus, $G(n,p)$ reaches unanimity on Day $6$ with probability at least $1-\varepsilon$.
\end{proof}

\section{Concluding remarks}\label{sec:conclude}

\textbf{Why $\bm{p\geq n^{-3/5}\log n}$?} 
The logarithm appears in the bound because of the Chernoff bound, which might be a purely technical reason. 
In fact, the statement of~\cref{thm:main} can be easily strengthened without too much effort. First, $p\geq \lambda' n^{-3/5}\log^{4/5}n$ suffices to guarantee probability $1-\varepsilon$ for unanimity to occur. Second, in \cref{sec:day0}, one may take $c$ that tends to $0$ slowly, e.g., $c=1/\log\log n$, to turn the main theorem into a w.h.p. statement too, while losing $(\log n)^{o(1)}$-factor in the lower bound for $p$. 
We however did not bother leaving the logarithmic factors as simple as it is, since the exponent $-3/5+o(1)$ does not seem to be tight even without the $o(1)$-factor.

The key technical bottleneck in improving the exponent $-3/5+o(1)$ is the use of Chebyshev's inequality to conclude the proof of \cref{lem:day0_main}.
We believe that our moments estimation is as accurate as possible except polylogarithmic factors. Thus, as long as one follows our proof outline and uses the Chebyshev inequality together with the Berry--Esseen bound, it may be difficult to improve the main term $-3/5$ in the exponent.

Even if one overcomes such technical obstacles and goes beyond $-3/5$, the next by far more challenging problem may be to reach beyond the exponent $-2/3$. Indeed, there are several points in our argument that uses $p\gg n^{-2/3}$, but most importantly, the shift of magnitude $pn^{3/2}$ given in~\cref{lem:day0_main} (and the same number in~\cref{lem:day2_bias} too) becomes void if $p\ll n^{-2/3}$. Overall, we suspect that improving the exponents $-3/5$ or $-2/3$ will require a substantially new approach.

\medskip

\noindent\textbf{The optimal initial bias.}
%It is also natural to ask for the optimal initial bias $\sum_{v}s_0(v)$ that guarantees unanimity in $G(n,p)$.
In~\cite{TV19}, Tran and Vu showed that $\sum_v s_0(v) =\Omega(1/p)$ is enough to guarantee unanimity to appear with probability $1-\varepsilon$ in majority dynamics on $G(n,p)$, for any $p\geq (2+o(1))(\log n)/n$, which generalizes the Fountoulakis--Kang--Makai theorem. 
Our result proves that, for $p=\Omega(n^{-3/5+o(1)})$,  the initial bias $\Omega(\sqrt{n})$ that can be smaller than $1/p=O(n^{3/5-o(1)})$ also suffices to guarantee the same conclusion. Furthermore, it also generalises to
\begin{theorem}
If $|\sum_{v}s_0(v)|=\Omega\left(n^{-1/4}p^{-5/4}\log n\right)$ and $\frac{\log^4 n}{n} \ll p\leq \lambda n^{-1/2}$, then majority dynamics on $G(n,p)$ admits unanimity with probability $1-\varepsilon$.
\end{theorem}
\noindent This improves the Tran--Vu theorem in the suggested range of $p$ and can be seen as a positive evidence for the conjecture by Berkowitz and Devlin~\cite[Conjecture 8]{BD20}, which states that the initial bias can be as small as one whenever $p\geq (1+o(1))(\log n)/n$.

\vspace{5mm}

\noindent\textbf{Acknowledgements.} Part of this work was carried out while the second and the third authors visited the other authors at IBS Daejeon.

\bibliographystyle{abbrv}
\bibliography{reference}

\end{document}